\documentclass[12pt,english]{amsart}
\usepackage[T1]{fontenc}
\usepackage[latin9]{inputenc}
\usepackage{geometry}
\usepackage{color}
\usepackage{babel}
\usepackage{refstyle}
\usepackage{amsmath}

\usepackage{amsthm}
\usepackage{amssymb}
\usepackage{esint}
\usepackage[unicode=true,
 bookmarks=true,bookmarksnumbered=false,bookmarksopen=false,
 breaklinks=false,pdfborder={0 0 0},
 colorlinks=true]
 {hyperref}

\makeatletter

\AtBeginDocument{\providecommand\eqref[1]{\ref{eq:#1}}}
\RS@ifundefined{subref}
  {\def\RSsubtxt{section~}\newref{sub}{name = \RSsubtxt}}
  {}
\RS@ifundefined{thmref}
  {\def\RSthmtxt{theorem~}\newref{thm}{name = \RSthmtxt}}
  {}
\RS@ifundefined{lemref}
  {\def\RSlemtxt{lemma~}\newref{lem}{name = \RSlemtxt}}
  {}

\newcommand{\strt}[1]{\rule{0pt}{#1}}

\theoremstyle{plain}
\newtheorem{thm}{\protect\theoremname}
\theoremstyle{plain}
\newtheorem{observation}{\protect\observationname}
\theoremstyle{plain}
\newtheorem*{observation*}{\protect\observationname}

  \theoremstyle{plain}
  \newtheorem{conjecture}{\protect\conjecturename}
  \theoremstyle{definition}
  \newtheorem{defn}{\protect\definitionname}
\newenvironment{lyxlist}[1]
{\begin{list}{}
{\settowidth{\labelwidth}{#1}
 \setlength{\leftmargin}{\labelwidth}
 \addtolength{\leftmargin}{\labelsep}
 }}
{\end{list}}
  \theoremstyle{plain}
  \newtheorem{lem}{\protect\lemmaname}

\newcommand{\note}[1]{\vskip.3cm
\fbox{%
\parbox{0.93\linewidth}{\footnotesize #1}} \vskip.3cm}

\def\XXint#1#2#3{{\setbox0=\hbox{$#1{#2#3}{\int}$}
    \vcenter{\hbox{$#2#3$}}\kern-.5\wd0}}

\DeclareMathOperator{\supp}{supp}
\usepackage{url}

\makeatother

  \providecommand{\conjecturename}{Conjecture}
  \providecommand{\definitionname}{Definition}
  \providecommand{\lemmaname}{Lemma}
  \providecommand{\observationname}{Observation}
\providecommand{\theoremname}{Theorem}

\begin{document}


\title[Three observations on commutators of SIO with BMO functions]{Three observations on commutators of Singular Integral operators with BMO functions}

\author{Carlos P\'erez}
\address{Carlos P\'erez, Department of Mathematics, University of the Basque Country UPV/EHU, 
IKERBASQUE, Basque Foundation for Science, Bilbao and BCAM, Basque Center for Applied Mathematics, Bilbao,  Spain.
}
\email{carlos.perezmo@ehu.es}

\author{Israel P. Rivera-R\'{\i}os}
\address{Israel P. Rivera-R\'{\i}os, IMUS \& Departamento de Análisis Matemático, Universidad de Sevilla, Sevilla, Spain}
\email{petnapet@gmail.com}



\thanks{The first author was supported by  Severo Ochoa Excellence Programme and the Spanish Government grant MTM2014-53850-P and the second author was supported by Grant MTM2012-30748, Spanish Government}


\begin{abstract}
Three observations on commutators of Singular Integral Operators with BMO functions are exposed, namely
\begin{list}{}{}
\item[\textbf{Section \ref{decay}}] The already known subgaussian local decay for the commutator, namely \[\frac{1}{|Q|}\left|\left\{x\in Q\, : \, |[b,T](f\chi_Q)(x)|>M^2f(x)t\right\}\right|\leq c e^{-\sqrt{ct\|b\|_{BMO}}}\] is sharp, since it cannot be better than subgaussian. 
\item[\textbf{Section \ref{sparse}}] It is not possible to obtain a pointwise control of the commutator by a finite sum of sparse operators defined by  $L\log L$ averages. 
\item[\textbf{Section \ref{negativeEstimateConjugationMethod}}] Motivated by the conjugation method for commutators, it is shown the failure of the following endpoint estimate, if $w\in A_p\setminus A_1$ then $$\left\| wM\left(\frac{f}{w}\right)\right\|_{L^1(\mathbb{R}^n)\rightarrow L^{1,\infty}(\mathbb{R}^n)}=\infty.$$
\end{list}
\end{abstract}

\maketitle

\section{Introduction}
The purpose of this paper is to present some observations concerning  commutators of singular integral operators with BMO functions. These operators were introduced by Coifman,
Rochberg and Weiss in \cite{MR0412721} 
as a tool to extend the classical factorization theorem for Hardy spaces in the unit circle to $\mathbb R^n$. These operators are defined by the expression
\begin{equation}
T_b f(x)=\int_{\mathbb R^n} (b(x)-b(y))K(x,y)f(y)\,dy,
\end{equation}
where $K$ is a kernel satisfying the standard Calder\'on-Zygmund
estimates and where $b$, the
``symbol'' of the operator, is a locally integrable function. Of course, these are special cases of the more 
general commutators given by the expression
$$T_b= [b,T]=M_{b}\circ T -T \circ M_{b}$$
where $T$ is any operator and $M_{b}$ is the multiplication operator $M_{b}f=b \cdot f$.


The classical well known result from \cite{MR0412721} establishes  that $[b, T]$ is a bounded operator on $L^{p}( \mathbb R^{n} )$, $1<p<\infty$, when the symbol $b$ is a $BMO$ function. We state this result.

\begin{thm} \label{Thm:LpBddness} 
Let $T$ be a singular integral operator and
$b$ a $BMO$ function. The commutator \,$T_b$ \,is bounded on $L^{p}(\mathbb{R}^{n})$ for every $1<p<\infty$.  
\end{thm}
%
In the same paper it is shown that \,$b\in BMO$\, is also a necessary condition namely, if the commutators $[b,R_{j}]$, $j=1,\cdots,n$\, of \,$b$\, with the Riesz transforms \,$R_{j}$\, are bounded on
\,$L^{p}(\mathbb{R}^{n})$\, for some $p\in(1,\infty)$ and every $j\in\{1,2,\dots,n\}$
then $b\in BMO$.


None of the different proofs of this result follows the usual
scheme of the classical Calder\'on-Zygmund theory for proving the $L^{p}(\mathbb{R}^{n})$ boundedness of singular
integral operators $T$. Two proofs of Theorem \ref{Thm:LpBddness} can be found in \cite{MR0412721}. The first and main one in that paper is based on methods involving techniques similar to those used in \cite{MR0380244} 
to understand the Calder\'on commutator. As far as we know this approach has not been so influential. However, the second proof, based on the so called conjugation method from operator theory, has been widely used. In fact, it is quite surprising that this proof was postponed to the end of the paper since it turns out to be highly interesting. Indeed, the method shows the intimate connection between these commutators and the $A_{p}$ theory of weights. Furthermore, this proof can be applied to general linear operators, not only for Singular Integral Operators. As a sample we will point out the following particular \,$L^{2}$\, case: 
\begin{thm}\label{Thm:Alt} Suppose that $T$ is a linear operator such that 
$$
T:L^{2}(w)\longrightarrow L^{2}(w)
$$
for every $w\in A_{2}$. Then for every $b\in BMO$, 
$$
[b,T]:L^{2}(\mathbb{R}^{n})\longrightarrow L^{2}(\mathbb{R}^{n}).
$$
\end{thm}
The method of proof can be carried out in more generality as shown in \cite{MR1211818}.
The key initial argument of the proof is that we can write $[b,T]$ as a complex integral operator using the Cauchy integral theorem
as follows 
$$
[b,T]f=\left.\frac{d}{dz}e^{zb}T(fe^{-zb})\right|_{z=0}=
\frac{1}{2\pi i}\int_{|z|=\varepsilon}
\frac{T_z(f)}{z^2}\,dz\, , \quad \varepsilon>0
$$
where 
$$z\to T_z(f) := e^{zb} T\left(\frac{f}{e^{zb}}\right) \qquad z \in \mathbb{C}. 
$$
This is called the ``conjugation'' of $T$ by $e^{zb}$ and the terminology comes most probably from group theory. Now,  if $\|\cdot\|$ is a norm we can apply Minkowski inequality:
$$
\left\| [b,T]f \right\| \leq \frac{1}{2\pi\,\varepsilon} \,\sup_{|z|=\varepsilon}  
\left\| T_z(f)\right\| \qquad \varepsilon>0. 
$$
The effectiveness of the method can be checked in the modern context of weighted $L^{p}$ estimates. Indeed, the method  produces very optimal bounds of the operator norm as shown in \cite{MR2869172} (see also \cite{MR3092729}).


This method reveals the role played by the following operation: 
\begin{equation*}
f
\to   T_w(f) := w\,T\left(\frac{f}{w}\right)
\end{equation*}
where $w$ is a weight which, in this context, is an $A_p$ weight. Indeed, this is the case by the well known key property of the BMO class, if $p>1$ and $b\in BMO$ then there is a small $\varepsilon_0$, such that $e^{tb}\in A_p$, for any real number $t$ such that $|t|<\varepsilon_0$.  These operators were already studied by B. Muckenhoupt and R. Wheeden in the 70's and  by E. Sawyer in the 80's. Some of the problems they left open were solved in   
\cite{MR2172941}. A consequence of the main result of \cite{MR2172941} is that if $w\in A_1$ then $T_w$ is of weak type $(1,1)$, namely 
$$
\left\| T_w  \right\|_{L^{1}(\mathbb{R}^{n}) \to L^{1, \infty}(\mathbb{R}^{n})}<\infty
$$
with bound depending upon the $A_1$ constant of $w$. However, we will exhibit examples of weights $w\in A_{p}\setminus A_{1}$ in Section \ref{negativeEstimateConjugationMethod} for which $T_w$ is {\bf not} of weak type $(1,1)$, namely
$$
\left\| T_w  \right\|_{L^{1}(\mathbb{R}^{n}) \to L^{1, \infty}(\mathbb{R}^{n})} =\infty.
$$
This shows that the case $w \in A_{1}$ is specially relevant. 
Perhaps, this  phenomenon can be explained by the fact that the conjugation method is closely attached to commutators with BMO functions which are {\bf not} of weak type $(1,1)$ as observed in \cite{MR1317714}. Indeed, the conjugation method works due to the property, already mentioned, that if $p>1$ and $b\in BMO$ then $e^{tb}\in A_p$ for small values of $t$. However,  this property turns out to be false in the case $p=1$.  The lack of the weak type $(1,1)$ property for commutators  is replaced by a $L\log L$ inequality like (\ref{LogL}) below and not better.  

There is another proof of Theorem \ref{Thm:LpBddness}  based on the use of the sharp maximal function of C. Fefferman and E. Stein which has also been very influential. 
It seems that it was first discovered by J. O. Str\"omberg as mentioned by S. Janson in 
\cite{MR524754} (see also \cite{T} pp. 417-419)
The proof relies on combining the following key pointwise estimate 
\begin{equation}\label{eq:St}
M^{\sharp}([b,T]f)\leq c\|b\|_{BMO}\left(M_{r}\,\left(Tf\right)+M_{s}\left(f\right)\right) 
\end{equation}
where $1<r,s<\infty$ and $M_r(f)= M(|f|^r)^{1/r}$ together with the classical Fefferman-Stein inequality: 
\[
\|M(f)\|_{L^{p}}\leq c\|M^{\sharp}(f)\|_{L^{p}}. 
\]
Here we use standard notation, $M$ is the Hardy-Littlewood maximal function and $M^{\sharp}$ is the sharp maximal function. 
The $L^{p}$ boundedness of $M$ and $T$ yields the alternative proof of Theorem \ref{Thm:LpBddness}.
Proceeding in the same way we obtain the corresponding estimates for
$A_{p}$ weights. 

This approach was considered by S. Bloom in (\cite{B}) extending in an interesting way Theorem \ref{Thm:LpBddness} but only on the real line.

\begin{thm}
Let  $\mu,\lambda\in A_{p}$ and let $H$ be the Hilbert transform: 
\[
[b,H]\,:\,L^{p}(\mu)\longrightarrow L^{p}(\lambda)
\]
where   $\nu=\mu^{\frac{1}{p}}\lambda^{-\frac{1}{p}}$
if and only if 
\begin{equation}\label{newBMO}
\|b\|_{BMO(\nu)}=\sup_{Q}\frac{1}{\nu(Q)}\int_{Q}\left|b-b_{Q}\right|<\infty.
\end{equation}

\end{thm}



The power of the pointwise estimate (\ref{eq:St}) is reflected in many situations, for instance in  \cite{MR1142721}, where similar results were derived for commutators of strongly singular integral with symbol in the new $BMO$ class (\ref{newBMO}) 
(see also \cite{HLW,2015arXiv150903769H}  for an alternative approach based on dyadic shifts).

However, estimate (\ref{eq:St}) is not sharp enough for many purposes and much better results can be obtained with the following variation:
\begin{equation} \label{EstimateJFA}
M_{\delta}^{\sharp}([b,T]f) \leq c\|b\|_{BMO}\left(M_{\varepsilon}\left(Tf\right)+M^{2}\left(f\right)\right) \quad 0<\delta<\varepsilon<1
\end{equation}
where $M^{2}$ stands for $M\circ M$ (see \cite{MR1317714}). 
Here, the key difference is that we are considering small parameters $\delta$ and $\varepsilon$.   
The estimate is sharp since $M^2$ cannot be replaced by the (pointwise) smaller operator $M$. Indeed,  otherwise these commutators would be of weak type $(1,1)$  but,  as we mentioned above, this is not the case \cite{MR1317714} where it is shown that commutators  satisfy the following  ``$L\log L$'' type estimate, 
%
\begin{equation} \label{LogL}
w\left(\left\{ x\in\mathbb{R}^{n}\,:\,\left|\left[b,T\right]f(x)\right|>\lambda\right\} \right)\leq c\,\int_{\mathbb{R}^{n}}\Phi\left(\frac{|f|}{\lambda}\|b\|_{BMO}\right)wdx \quad \lambda>0,
\end{equation}
where $w\in A_{1}$, $\Phi(t)=t\log\left(e+t\right)$ and where $c>0$ depends upon the $A_1$ constant. This shows that these commutators are 
``more singular'' than Calder\'on-Zygmund operators. 
The original proof of (\ref{LogL}) follows from the key pointwise (\ref{EstimateJFA}) combined with a good-$\lambda$ type argument,  but an alternative proof was obtained by the first author and G. Pradolini in  \cite{MR1827073} with the bonus that non $A_{\infty}$ weights can be considered. 
This argument is based on a variation of the classical scheme used to prove the weak type $(1,1)$ for Calderón-Zygmund operators. 
The statement of the result is the following. 
\begin{thm}
\label{Thm:EndPCual}Let $T$ be a Calderón-Zygmund
operator and $b\in BMO$. If $w$ is an arbitrary weight the following
inequality holds
\[
w\left(\left\{ x\in\mathbb{R}^{n}\,:\,\left|[b,T]f(x)\right|>\lambda\right\} \right)\leq C_{\varepsilon,T}\int_{\mathbb{R}^{n}} \Phi\left(\|b\|_{BMO}\frac{|f(x)|}{\lambda}\right)M_{L\left(\log L\right)^{1+\varepsilon}}w(x)dx
\]
for every $\varepsilon>0$.
\end{thm}

Very recently (c.f. \cite{2015arXiv150708568P}) the authors have obtained
a quantitative version of the endpoint estimate for arbitrary weights,
namely Theorem \ref{Thm:EndPCual}. This result is analogous to the
one obtained by the first author and T. Hytönen for singular integrals in \cite{MR3092729}.
\begin{thm}
Let $T$ be a Calderón-Zygmund operator
and $b\in BMO$. If $w\geq0$ is a weight then, for every $\varepsilon>0$
\[
w\left(\left\{ x\in\mathbb{R}^{n}\,:\,|[b,T]f(x)|>\lambda\right\} \right)\leq \frac{c}{\varepsilon^{2}}\int_{\mathbb{R}^{n}}\Phi\left(\|b\|_{BMO}\frac{|f|}{\lambda}\right)M_{L\left(\log L\right)^{1+\varepsilon}}wdx.
\]

\end{thm}

The main novelty here is the appearance of the sharp factor \,$ \frac{1}{\varepsilon^{2}}$\,  reflecting again the higher singularity of the operator.    As a corollary of this result we can derive the following result obtained previously by C. Ortiz-Caraballo in \cite{MR3008263}, 
$$ 
w\left(\left\{ x\in\mathbb{R}^{n}\,:\,|[b,T]f(x)|>\lambda\right\} \right)\leq C\Phi\left([w]_{A_{1}}\right)^{2} 
\int_{\mathbb{R}^{n}} 
\Phi\left (\|b\|_{BMO}\frac{|f|}{\lambda} \right) \,wdx. 
$$

We remark that it seems that the conjugation method cannot be applied to prove this estimate. Therefore, estimate (\ref{LogL}) or Theorem \ref{Thm:EndPCual} works, so far, for Calder\'on-Zygmund operators not for general linear operators assuming a minimal appropriate weighted weak type estimate.

Another interesting difference between Calder\'on-Zygmund operators and commutators concerns their local behavior. A very nice way of expressing this is by means of the following estimate due to Karagulyan \cite{MR1913610}: there exists a constant $c>0$ such that for each cube $Q$  and for each function $f$ supported on the cube $Q$ 
\begin{equation}\label{Subgausiandecay}
\frac{1}{|Q|} |\{x\in Q: |Tf(x)|> tMf(x)\}|\le c\, e^{-c\, t}\qquad t>0.
\end{equation}
This result can be seen as an improvement of Buckley's exponential decay theorem \cite{MR1124164} which is a very useful result.  For instance, it allows to improve in a quantitative way the classical good-$\lambda$ inequality between $T$ and $M$: if $p\in (0,\infty)$ and $w\in A_{\infty}$
\[
\|Tf\|_{L^{p}(w)}\leq c_T\,p\,[w]_{{A_{\infty}}} \|M(f)\|_{L^{p}(w)}. 
\]

Motivated by this result of Karagulyan, Ortiz-Caraballo,
Rela and the first author developed a new method for proving  (\ref{Subgausiandecay}) in \cite{MR3124931}. This method is flexible enough to deal 
with other operators including the commutators. In particular, we have the following sub-gaussian estimate.  

\begin{thm}\label{OPRthm}
Let $T$ be a Calder\'on-Zygmund operator and $b\in BMO$, then there exists  a constant $c>0$ 
such that for each $f$ 
\begin{equation}
\sup_{Q}\frac{1}{|Q|}|\{x\in Q:|[b,T](f\chi_Q)(x)|>tM^{2}f(x)\}|\leq c\,e^{-\sqrt{c\,t\|b\|_{BMO}}}\qquad t>0.\label{OPRestimate}
\end{equation}
\end{thm}

We will show in Section \ref{decay} that this subexponential decay is fully sharp. In Section \ref{sparse}, we will provide a new proof of (\ref{Subgausiandecay}) based on the pointwise domination: if $T$ is a Calder\'on-Zygmund operator, then it is 
possible to find a finite set of $\eta$-sparse families $\left\{ \mathcal{S}_{j}\right\} _{j=1}^{3^n}$ (see Section \ref{sparse} for the definitions) 
contained in the same or in different dyadic lattices $\mathcal{D}_{j}$
and depending on $f$ such that 
\begin{equation}\label{DominationTheorem}
|Tf(x)|\leq c_{T}\sum_{j=1}^{3^n}A_{\mathcal{S}_{j}}f(x) 
\end{equation}
where
$$A_{\mathcal{S}_{j}}f(x)=\sum_{Q\in\mathcal{S}_{j}}   \frac{1}{|Q|} \int_{Q} |f| \,\chi_{Q}(x).
$$
See Section \ref{sparse} for details,  in particular Theorem \ref{DominationThm}.

In view of the interest of an estimate like (\ref{DominationTheorem}) it would be relevant to produce a  counterpart for commutators. The ``natural'' sparse operator for these commutators would be
\[
B_{\mathcal{S}}f(x)=\sum_{Q\in\mathcal{S}}\, \|f\|_{\strt{1.7ex} L\log L,Q}\, \chi_{Q}(x).
\]
The reason that leads to consider this sparse operator in terms of the average\\ $\| \cdot \|_{\strt{1.7ex} L\log L,Q}$ is due to the intimate relationship of commutators and $M^2$ which is an operator pointwise equivalent to $M_{\strt{1.5ex} L\log L}$. In Section \ref{sparse} we prove the impossibility of having a domination theorem for commutators by these  ``sparse'' operators.


%

\section{First observation:  Sharpness of the subexponential local decay} \label{decay}

We prove in this section that Theorem \ref{OPRthm} is sharp, i.e., we can
find a Calderón-Zygmund operator $T$, a symbol $b\in BMO$ a function
$f$ and a cube $Q$ such that 
\[
\frac{1}{|Q|}|\{x\in Q:|[b,T]f(x)|>tM^{2}f(x)\}|\geq c\,e^{-\sqrt{c\,t\|b\|_{BMO}}}
\]
for some constant $c>0$. More precisely we have the following.


\begin{observation}
\label{Thm:SharpLocal}Let $b(x)=\log|x|$, then  we can find a constant $c>0$ such that 
\[
|\{x\in(0,1):|[b ,H](\chi_{(0,1) })(x) |>t\}|\geq e^{-\sqrt{ct}}
\]
where $H$ stands for the Hilbert transform.
\end{observation}
\begin{proof}
Let $f(x)=\chi_{(0,1)}(x).$ We are going to show that 
\[
|\{x\in(0,1):|[b,H]f(x)|>tM^{2}f(x)\}|=|\{x\in(0,1):|[b,H]f(x)|>t\}|\geq c\,e^{-\sqrt{\alpha\,t}}\qquad t>0.
\]
For $x\in(0,1)$ we have that 
\[
[b,H]f(x)=\int_{0}^{1}\frac{\log(x)-\log(y)}{x-y}\,dy=\int_{0}^{1}\frac{\log(\frac{x}{y})}{x-y}\,dy=\int_{0}^{1/x} \frac{\log(\frac{1}{t})}{1-t}\,dt.
\]
Now we observe that 
\[
\int_{0}^{1/x}\frac{\log(\frac{1}{t})}{1-t}\,dt=\int_{0}^{1}\frac{\log(\frac{1}{t})}{1-t}\,dt+\int_{1}^{1/x}\frac{\log(\frac{1}{t})}{1-t}\,dt
\]
and since $\frac{\log(\frac{1}{t})}{1-t}$ is positive for $(0,1)\cup(1,\infty)$
we have for $0<x<1$ that 
\[
|[b,H]f(x)|>\int_{1}^{1/x}\frac{\log(\frac{1}{t})}{1-t}\,dt.
\]
Finally, a computation shows that 
$$
\int_{1}^{1/x}\frac{\log(\frac{1}{t})}{1-t}\,dt  \approx \left(\log\frac{1}{x}\right)^{2} \qquad x\to 0.
$$
%
Consequently, we have that for some $x_{0}<1$ 
\[
|[b,H]f(x)|>c\,\left(\log\frac{1}{x}\right)^{2}\qquad0<x<x_{0}.
\]
and then for some $t_0>0$,  
\begin{equation}\label{cotafinal}
\left| \{x\in(0,1):|[b,H]f(x)|>t\} \right| \geq  \left|\left\{x\in(0,x_{0}):c\,\left(\log\frac{1}{x}\right)^{2}>t\right\}\right|=e^{-\sqrt{t/c}} \qquad t>t_{0} 
\end{equation}
as we wanted to prove.
\end{proof}

\section{Second observation: a ``natural'' but false sparse domination result for commutators} \label{sparse}


Before stating the result we are going to prove in this section we
need some notation. We borrow it from \cite{2015arXiv150805639L}.
\begin{defn}
[Dyadic child]Let $Q$ be a cube (with sides parallel to the axis).
We call dyadic child any of the $2^{n}$ cubes obtained by partitioning
$Q$ by $n$ ``median hyperplanes'' (planes parallel to the faces
of $Q$ and dividing each edge into 2 equal parts).
\end{defn}
If we iterate the partition process of the preceding definition we
obtain a standard dyadic lattice $\mathcal{D}(Q)$ of subcubes of
$Q$ which has the usual properties:
\begin{enumerate}
\item For each $k=0,1,2,\dots$ cubes in the $k$-th generation have sidelength
$2^{-k}$ and tile $Q$ in a regular way.
\item Each $Q'$ in the $k$-th generation has $2^{n}$ children in the
in the $(k+1)$-th generation contained in it and one and only one
parent in the $(k-1)$-th generation containing it (unless it is $Q$
itself).
\item If $Q',Q''\in\mathcal{D}(Q)$, then $Q'\cap Q''=\emptyset$ or $Q'\subseteq Q''$
or $Q''\subseteq Q'$.
\item If $Q'\in\mathcal{D}(Q)$, then $\mathcal{D}(Q')\subseteq\mathcal{D}(Q)$. \end{enumerate}
\begin{defn}
[Dyadic lattice]A dyadic lattice $\mathcal{D}$ in $\mathbb{R}^{n}$
is any collection of cubes such that 
\begin{lyxlist}{00.00.0000}
\item [{(DL-1)}] If $Q\in\mathcal{D}$ then each dyadic child of $Q$
is in $\mathcal{D}$ as well.
\item [{(DL-2)}] If $Q',Q''\in\mathcal{D}$ there exists $Q\in\mathcal{D}$
such that $Q',Q''\in\mathcal{D}(Q)$.
\item [{(DL-3)}] If $K$ is a compact set of $\mathbb{R}^{n}$ there exists
$Q\in\mathcal{D}$ such that $K\subseteq Q$.
\end{lyxlist}
\end{defn}
There is an easy way to build a dyadic lattice by considering a
increasing sequence of dyadic cubes $Q_{j}$ such that $\cup_{j=1}^{\infty}Q_{j}=\mathbb{R}^{n}$.
Then 
\[
\mathcal{D}=\bigcup_{j=1}^{\infty}\mathcal{D}(Q_{j})
\]
is a dyadic lattice.

\begin{defn}\label{DefE(Q}
Let $\eta\in(0,1)$. We say that a family of cubes $\mathcal{S}\subseteq\mathcal{D}$
is $\eta$-sparse if for each $Q\in\mathcal{S}$ we can find a measurable
subset $E(Q)\subset Q$ such that:
\begin{enumerate}
\item $E(Q)$'s are pairwise disjoint.
\item $\eta|Q|\leq|E(Q)|$
\end{enumerate}
\end{defn}

\begin{defn}
Let $\Lambda>1$. We say a family of cubes $\mathcal{S}$ is $\Lambda$-Carleson
if for every cube $Q\in\mathcal{D}$ we have 
\[
\sum_{P\in\mathcal{S},P\subset Q}|P|\leq\Lambda|Q|.
\]

\end{defn}
There is an interesting relation between Carleson and sparse families
that we summarize in the following lemma
\begin{lem}\label{lem:SparseCarleson}
If $\mathcal{S}$ is a $\Lambda$-Carleson family of cubes then it
is $\frac{1}{\Lambda}$-sparse. Conversely if $\mathcal{S}$ is a
$\eta$-sparse family of cubes then it is a $\frac{1}{\eta}$-Carleson
family of cubes.
\end{lem}
Armed with all these definitions we can state  the following pointwise domination theorem.

\begin{thm} \label{DominationThm}
Let $T$ be a Calderón-Zygmund operator. There is a finite set of $\eta$-sparse families $\left\{ \mathcal{S}_{j}\right\} _{j=1}^{3^n}$
contained in the same or in different dyadic lattices $\mathcal{D}_{j}$
and depending on $f$ such that 
\begin{equation} 
T^{*}f(x)\leq c_{T,n}\sum_{j=1}^{3^n}A_{\mathcal{S}_{j}}f(x) 
\end{equation}
where $A_{\mathcal{S}_{j}}f(x)=\sum_{Q\in\mathcal{S}_{j}}   \frac{1}{|Q|}\int_Q |f| \, \chi_{Q}(x)$.
\end{thm}

The proof of this result can be found in \cite{2015arXiv150805639L} and \cite{2014arXiv1409.4351C}. In \cite{2015arXiv150105818L} M. Lacey obtains the same estimate for Calderón-Zygmund operators that satisfy a Dini condition. Recently a fully quantitative version of Lacey's result was obtained in \cite{2015arXiv151005789H} and even more recently this quantitative version has been simplified in \cite{2015arXiv151207247L}.

As a sample of the interest of this result we give a different proof of the exponential estimate (\ref{Subgausiandecay}):  
there exists a constant $c>0$ such that for each cube $Q$  and for each $f$ supported on the cube $Q$ 
\begin{equation}\label{Subgausiandecay2}
\frac{1}{|Q|}\left |\left\{x\in Q: T^*f(x)> tMf(x)\right\}\right| \leq c\, e^{-c\, t}\qquad t>0,
\end{equation}

To prove this result we will use the classical \emph{vector-valued extension of the maximal function} introduced by Fefferman and Stein in \cite{MR0284802} that can be written as follows:
$$\overline{M}_qf(x)=\Big( \sum_{j=1}^{\infty} (Mf_j(x))^q \Big)^{1/q}=|Mf(x)|_q,$$
where ${f}=\{f_j\}_{j=1}^{\infty}$ is a vector--valued function.

Taking into account that $T^*$ is controlled by a finite sum of sparse operators it suffices to establish (\ref{Subgausiandecay2}) for those operators. 

Assume that $\supp f \subseteq Q$ for a certain cube. It is clear that we can find $c_n$ pairwise disjoint cubes $Q_j\in\mathcal{D}$ which union covers  $Q$ and such that $|Q_j| \simeq |Q|$ We can assume those cubes to belong to any sparse family $\mathcal{S}\subset\mathcal{D}$, since it's easy to check, taking into account Lemma \ref{lem:SparseCarleson}, that adding a finite number of pairwise disjoints cubes to a sparse family the resulting family is again a sparse family.

First we are going to prove that if $\mathcal{S}$ is and $Q_j\in\mathcal{S}$ with $|Q_j|\simeq |Q|$, as we have just showed that we can assume, then
\begin{equation}
\frac{1}{|Q|}\left|\left\{x\in Q:   \sum_{P\in\mathcal{S}, \, P\subseteq{Q_j}} \chi_{P}(x)> t \right\}\right| \leq ce^{-ct}\label{ClaimLocal}
\end{equation}

We begin observing that 
\begin{eqnarray*}
&&\frac{1}{|Q|}\left|\left\{x\in Q:   \sum_{P\in\mathcal{S}, \, P\subseteq{Q_j}} \chi_{P}(x)> t \right\}\right| \\
&=&\frac{1}{|Q|}\left|\left\{x\in Q\cap Q_j:   \sum_{P\in\mathcal{S}, \, P\subseteq{Q_j}} \chi_{P}(x)> t \right\}\right| \\
& \le & c\frac{1}{|Q_j|}\left|\left\{x\in Q_j:   \sum_{P\in\mathcal{S}, \, P\subseteq{Q_j}} \chi_{P}(x)> t \right\}\right| =C_Q
\end{eqnarray*}

We use now one of  the key estimates from \cite{MR3124931}. Indeed, let $\{E(P)\}_{{P\in\mathcal{S}, \, P\subseteq{Q_j}}}$ be the family of sets from Definition \ref{DefE(Q}. We have then for some $c>0$ that
\begin{eqnarray*}
\sum_{P\in\mathcal{S}} \chi_{P}(x)&=&\sum_{Q\in\mathcal{S}}   \left({\frac{1}{|P|}}\,|P|\right)^q\chi_{P}(x)\\
&\leq& c\,\sum_{P\in\mathcal{S}, \, P\subseteq{Q_j}}   \left({\frac{1}{|P|}}\,|E(P)|\right)^q\chi_{P}(x)\\
&\leq& c\,\sum_{P\in\mathcal{S}, \, P\subseteq{Q_j}} \left({\frac{1}{|P|}}\,\int_{P}\chi_{E(P)}(y)\,dy\right)^q\chi_{P}(x)\\
&\leq& c\, \left(\overline{M}_q\left(\left\{\chi_{E(P)}\right\}_{P\in\mathcal{S}, \, P\subseteq{Q_j}}\right)(x)\right)^q\\
&\leq& c\, \left(\overline{M}_qg_j(x)\right)^q,
\end{eqnarray*}
where $g_j=\left\{\chi_{E(P)}\right\}_{P\in\mathcal{S}, \, P\subseteq{Q_j}}$ is supported in $Q_j$.  Now, since $\{E(Q)\}_{P\in\mathcal{S}, \, P\subseteq{Q_j}}$ is a pairwise disjoint family of subsets, we have that for any $j$
\begin{equation}
\|g_j(x)\|_{\ell^{q}}=\left(\sum_{P\in\mathcal{S}, \, P\subseteq{Q_j}}\left(\chi_{E(Q)}(x)\right)^q\right)^{1/q}\leq 1.
\end{equation}
We finish the proof of (\ref{ClaimLocal}) recalling that if $|g_j|_{\ell^{q}}\in L^{\infty}$, then $\left(\overline{M}_qg_j(x)\right)^q\in Exp L$ (see \cite{MR0284802}) from  which we conclude that:
\begin{equation*}
C_Q\leq c e^{-c t},\qquad t>0.
\end{equation*}
Now we go back to the proof of the estimate. We first observe that 
\[\begin{split}&\frac{1}{|Q|}\left |\left\{x\in Q: A_\mathcal{S}f(x)> tMf(x)\right\}\right| \\ \leq &  \sum_{j=1}^{c_n}\frac{1}{|Q|}\left |\left\{x\in Q: A_{\mathcal{S}}\left(f\chi_{Q_j}\right)(x)> \frac{t}{c_n}M(f\chi_{Q_j})(x)\right\}\right|. \end{split}\]

Hence it suffices to obtain an estimate for each term of the sum. First we may assume that $|Q_j\cap Q|\neq0$ since otherwise, $\int_E f\chi_{Q_j}=0$ for every measurable set and the corresponding term in the sum equals zero.  Now we split $A_\mathcal{S}(f\chi_{Q_j}) $ as follows

\begin{equation}\label{SplitSparse}
A_\mathcal{S}(f\chi_{Q_j}) (x)=\sum_{P\in\mathcal{S},\, P\subsetneq Q_j} \frac{1}{|P|}\int_{P}f\chi_{Q_j}\chi_P(x) + \sum_{P\in\mathcal{S},\, P\supseteq Q_j} \frac{1}{|P|}\int_{P}f\chi_{Q_j}\chi_P(x) .
\end{equation}

We observe that for the first term

\begin{equation*}
\frac{\sum_{P\in\mathcal{S},\, P\subsetneq Q_j} \frac{1}{|P|}\int_{P}f\chi_{Q_j}\chi_P(x)}{M\left(f\chi_{Q_j}\right)(x)}\leq \sum_{P\in\mathcal{S},\, P\subseteq Q_j} \chi_P(x)
\end{equation*}

For the second term, we have that that since $Q_j\cap Q\neq\emptyset$ and $|Q|\simeq |Q_j|$ then $Q\subset 5Q_j$. Consequently,
\begin{eqnarray*}
& \frac{\sum_{P\in\mathcal{S},\, P\supseteq Q_j} \frac{1}{|P|}\int_{P}f\chi_{Q_j}\chi_P(x)}{M(f\chi_{Q_j})(x)}
&\leq  \frac{\sum_{P\in\mathcal{S},\, P\supseteq Q_j} \frac{1}{|P|}\int_{Q_j}f\chi_P(x)}{\frac{1}{|5Q_j|}\int_{Q_j} f}\\
& &\leq c_n\sum_{P\in\mathcal{S},\, P\supseteq Q_j} \frac{|Q_j|}{|P|}\chi_P(x)\\
& &\leq c_n\sum_{k=0}^\infty \frac{1}{2^{nk}}\end{eqnarray*}

Then, combining the estimates obtained for each of the terms of (\ref{SplitSparse}), we have that

\[\begin{split}&\frac{1}{|Q|}\left |\left\{x\in Q: A_{\mathcal{S}}\left(f\chi_{Q_j}\right)(x)> \frac{t}{c_n}M(f\chi_{Q_j})(x)\right\}\right|\\
\leq&\frac{1}{|Q|}\left |\left\{x\in Q: \sum_{P\in\mathcal{S},\, P\subseteq Q_j} \chi_P(x)  > \frac{t}{c_n}-c_n\sum_{k=0}^\infty \frac{1}{2^{nk}}\right\}\right|\end{split}\]

and the desired conclusion, namely (\ref{Subgausiandecay2}), follows from (\ref{ClaimLocal}).


\begin{observation}
Let $T$ be a Calder\'on-Zygmund operator and $b\in BMO$. It is not
possible to find a finite set of $\eta$-sparse families $\left\{ \mathcal{S}_{j}\right\} _{j=1}^{N}$, with $N$ dimensional, 
contained in the same or in different dyadic lattices $\mathcal{D}_{j}$
and depending on $f$ such that 
\begin{equation} \label{eq:FakeSparse}
|[b,T]f(x)|\leq c_{b,T}\sum_{j=1}^{N}B_{\mathcal{S}_{j}}f(x) \qquad a.e. \,\,x \in \mathbb{R}^{n}
\end{equation}
where $B_{\mathcal{S}_{j}}f(x)=\sum_{Q\in\mathcal{S}_{j}} \|f\|_{\strt{1.7ex} L\log L,Q}\, \chi_{Q}(x)$.
\end{observation}
We are going to give two proofs of this result. The first one is
based on the Rubio de Francia algorithm.
\begin{proof}
[Proof 1] Suppose that (\ref{eq:FakeSparse}) holds, then we can prove the following $L^{1}$ inequality
\begin{equation}
\|[b,T]f\|_{L^{1}(w)}\leq c[w]_{A_{1}}\|M^{2}f\|_{L^{1}(w)}.\label{eq:A1fake}
\end{equation}
Indeed,
\[
\begin{split}\|[b,T]f\|_{L^{1}(w)} & \leq c_{b,T}\sum_{j=1}^{N}\|B_{\mathcal{S}_{j}}f\|_{L^{1}(w)}\\
 & \leq c_{b,T}\sum_{j=1}^{N}\sum_{Q\in\mathcal{S}_{j}} \|f\|_{\strt{1.7ex} L\log L,Q} \frac{w(Q)}{|Q|}|Q|\\
 & \leq\frac{c_{b,T}}{\eta}\sum_{j=1}^{N}\sum_{Q\in\mathcal{S}_{j}} \|f\|_{\strt{1.7ex} L\log L,Q} \frac{w(Q)}{|Q|}|E(Q)|\\
 & \leq\frac{c_{b,T}}{\eta}\sum_{j=1}^{N}\sum_{Q\in\mathcal{S}_{j}}\int_{E(Q)}M_{\strt{1.7ex} L\log L}f(x)Mw(x)dx\\
 & \leq N\frac{c_{b,T}}{\eta}[w]_{A_{1}}\|M^{2}f\|_{L^{1}(w)},
\end{split}
\]
since $M^2 \approx M_{\strt{1.5ex} L\log L}$.  We claim now the $L^p$ version,
\begin{equation}
\|[b,T]f\|_{L^{p}(\mathbb{R}^{n})}\leq c_{n}p\|M^{2}f\|_{L^{p}(\mathbb{R}^{n})}\qquad p>1.\label{eq:int}
\end{equation}
Indeed, by duality we can find $g\geq0$ in $L^{p'}(\mathbb{R}^{n})$ with
unit norm such that 
\[
\|[b,T]f\|_{L^{p}(\mathbb{R}^{n})}=\int_{\mathbb{R}^{n}}|[b,T]f(x)|g(x)dx.
\]
We consider the Rubio de Francia algorithm
\[
Rg=\sum_{k=0}^{\infty}\frac{M^{k}(g)}{\|M\|_{L^{p'}(\mathbb{R}^{n})}^{k}}.
\]
It's a straightforward computation that $R(g)$ is an $A_{1}$ weight
with constant 
\[
\left[Rg\right]_{A_{1}}\leq2\|M\|_{L^{p'}}\leq c_{n}p
\]
and also that $g\leq Rg$ and $\|Rg\|_{L^{p'}}\leq2\|g\|_{L^{p'}(\mathbb{R}^{n})}=2$.
Then have that 
\[
\int_{\mathbb{R}^{n}}|[b,T]f(x)|g(x)dx\leq\int_{\mathbb{R}^{n}}|[b,T]f(x)|Rg(x)dx
\]
and using \eqref{A1fake} and Hölder inequality 
\[
\begin{split} & \int_{\mathbb{R}^{n}}|[b,T]f(x)|Rg(x)dx\leq c\left[Rg\right]_{A_{1}}\int_{\mathbb{R}^{n}}M^{2}f(x)Rg(x)dx\\
 & \leq cp\int_{\mathbb{R}^{n}}M^{2}f(x)Rg(x)dx\leq cp\|M^{2}f\|_{L^{p}(\mathbb{R}^{n})}\|Rg\|_{L^{p'}(\mathbb{R}^{n})}\\
 & \leq cp\|M^{2}f\|_{L^{p}(\mathbb{R}^{n})}.
\end{split}
\]
Hence \eqref{int} is established. Now since 
\[
\|M^{2}\|_{L^{p}(\mathbb{R}^{n})}\leq c_n\,\left(p'\right)^{2}\qquad p>1
\]
we have that 
\begin{equation}\label{eq:BadUpperBound} 
\|[b,T]\|_{L^{p}(\mathbb{R}^{n})}\leq cp\left(p'\right)^{2}  \qquad p>1
\end{equation}
Now let us observe that if we take $[b,H]f$ with $b(x)=\log|x|$
and $f(x)=\chi_{(0,1)}(x)$ then 
\[
\|[b,H]f\|_{L^{p}(\mathbb{R})}\geq cp^{2} \qquad p>1, 
\]
and this leads to a contradiction when $p\rightarrow\infty$. To prove this lower estimate we use estimate (\ref{cotafinal}) from Theorem
\ref{Thm:SharpLocal}. Indeed, for some $t_0>0$
\[
\begin{split}\|[b,H]f\|_{L^{p}(\mathbb{R})} & \geq\|\left[b,H\right]f\|_{L^{p,\infty}(\mathbb{R})}=\sup_{t>0}t|\{x\in\mathbb{R}:|[b,H]f(x)|>t\}|^{\frac{1}{p}}\\
 & \geq\sup_{t>t_{0}}t\left|\left\{ x\in(0,x_{0}):\,c\left(\log\frac{1}{x}\right)^{2}>t\right\} \right|^{\frac{1}{p}}\\
 & \geq\sup_{t>t_{0}}tce^{\frac{-\sqrt{t}}{p}} \geq c\,p^2\,t_0   e^{-\sqrt{t_0} }
 %
\end{split}
\]
and this concludes the first proof.
\end{proof}
For the second proof we will rely on the sharpness
result that was settled in the previous section, namely, we are going to prove that if a pointwise sparse control as the one in (\ref{eq:FakeSparse}) holds, then the commutator would have a local exponential decay, which as we established before, is not the case.
\begin{proof}
[Proof 2] Assume again that (\ref{eq:FakeSparse}) holds. Then, for some $c>1$
\[
\begin{split} & \left|\left\{ x\in Q:|[b,T]f(x)|>tM^{2}f(x)\right\} \right|\leq 
\left|\left\{ x\in Q:   \sum_{j=1}^{N} \sum_{P\in\mathcal{S}_j }\|f\|_{\strt{1.7ex} L\log L,P}\chi_{P}(x)>\frac{t}{c}\,M^{2}f(x)\right\} \right|\end{split}
\]

It will be enough for our purposes to work on each term of the inner sum, namely to control 
\[\left|\left\{ x\in Q:   \sum_{P\in\mathcal{S}_j }\|f\|_{\strt{1.7ex} L\log L,P}\chi_{P}(x)>t M^{2}f(x)\right\}\right|\]
Now, recalling that $M^{2}f\simeq M_{\strt{1.7ex} L\log L}f$, is not hard to see that essentially the same argument we used to prove (\ref{Subgausiandecay2}) yields that
\[
\frac{1}{|Q|}\left|\left\{ x\in Q:   \sum_{P\in\mathcal{S}_j }\|f\|_{\strt{1.7ex} L\log L,P}\chi_{P}(x)>t M^{2}f(x)\right\}\right| \leq ce^{-\alpha t}
\]
Combining the preceding estimates we arrive to 
\[
\frac{1}{|Q|}\left|\left\{ x\in Q:|[b,T]f(x)|>tM^{2}f(x)\right\} \right|\leq ce^{-\alpha t} \qquad   t>0
\]
which is a contradiction by Observation \ref{Thm:SharpLocal}. 
\end{proof}
The correct pointwise control for the commutator seems to be the following
one
\begin{conjecture}
Let $T$ be a Calderón-Zygmund operator and $b\in BMO$. Then 
\[
|[b,T]f(x)|\leq C(n,T)\|b\|_{BMO}\sum_{i,j=1}^{N}A_{\mathcal{S}_{i}}\left(A_{\mathcal{S}_{j}}f\right)(x)
\]
where $A_{\mathcal{S}_{j}}f(x)=\sum_{Q\in\mathcal{S}_{j}}\frac{1}{|Q|}\int_{Q}\left|f(y)\right|dy\chi_{Q}(x)$
and the sparse families $\mathcal{S}_{j}$ are not necessarily subfamilies of
the same dyadic lattice.
\end{conjecture}
If this conjecture holds it would be very easy to recover the main theorem from \cite{MR2869172} since it suffices to iterate the following estimate:
\[
\left\Vert A_{\mathcal{S}_{j}}f\right\Vert _{L^{p}(w)}\leq C_{n,p}[w]_{A_{p}}^{\max\left\{ 1,\frac{1}{p-1}\right\} }\|f\|_{L^{p}}.
\]
which was studied in \cite{MR2628851,MR2854179} (see also \cite{2015arXiv150805639L}).

\section{Third observation: The failure of a endpoint estimate motivated by the
conjugation method} \label{negativeEstimateConjugationMethod}

In this section we consider the following family of operators: 
\begin{equation}
f
\to   T_w(f) := w\,T\left(\frac{f}{w}\right)
\end{equation}
where $w$ is a weight and $T$ is a Calder\'on-Zygmund operator. We already mentioned in the introduction that these operators are of interest since they are very much related  to commutators due to the conjugation method. We emphasized that the case $w\in A_1$ is special since $T_w$ is of weak type $(1,1)$  as a consequence of the main results  from \cite{MR2172941}. Understanding the case $w\in A_p$ would be more interesting 
due to its connection with the conjugation method.  However, $T_w$ is {\bf not} of weak type $(1,1)$ in general since there are weights $w\in A_{p}\setminus A_{1}$ for which 
$$
\left\| T_w  \right\|_{L^{1}(\mathbb{R}^{n}) \to L^{1, \infty}(\mathbb{R}^{n})} =\infty,
$$
being the purpose of this section to show the existence of such weights. In fact we are going to show something worst replacing $T$ by the less singular operator $M$.

\begin{observation} \label{Thm19}   Let $1<p<\infty$, then there is \,$w\in A_{p}\setminus A_{1}$\, such that
$$
\left\Vert M_w  \right\Vert_{L^{1}(\mathbb{R}^{n}) \to L^{1, \infty}(\mathbb{R}^{n})} =\infty.
$$

\end{observation}

\begin{proof}
In dimension $1$ we choose the $A_p$ weight $w(x)=|x|^{-\delta(1-p)}$ with $\delta\in\left(0,\min\left\{ 1,\frac{1}{p-1}\right\} \right)$
and $f=\chi_{[0,1]}$ so that  $f\in L^{1}(w)$.  We prove that 
\[
\left\Vert wM\left(\frac{f}{w}\right)\right\Vert _{L^{1,\infty}\left(\mathbb{R}\right)}=\infty.
\]
Indeed, a computation shows that for $x>1$
\[
M\left(\frac{\chi_{(0,1)}}{w}\right)(x)\geq \frac{1}{x}\frac{1}{\beta}
\]
with  $\beta=1+\delta(1-p)$ and then 
\[
\begin{split}\left\Vert wM\left(\frac{f}{w}\right)\right\Vert _{L^{1,\infty}\left(\mathbb{R} \right)} 
 & \geq \frac{1}{\beta}\, \sup_{t>0}t\left|\left\{ x>1\,:  \,x^{-\delta(1-p)-1}>t \right\} \right|  = \frac{1}{\beta}\, \sup_{1>t>0}t\left(\left(\frac{1}{t}\right)^{\frac{1}{\beta}}-1\right)=\infty
\end{split}
\]
since $\beta \in (0,1)$.

\end{proof}

An interesting question is to find a necessary and sufficient condition for the boundedness  of this operator, namely, characterize the weights $w$ for which 
$$ \left\Vert M_w  \right\Vert_{L^{1}(\mathbb{R}^{n}) \to L^{1, \infty}(\mathbb{R}^{n})} < \infty.
$$
In \cite{MR0447956} Muckenhoupt and Wheeden proved that this inequality
holds for $w\in A_{1}$ in the real line and also obtained a necessary
condition on the weights, namely
\[
\left\Vert \frac{w\chi_{Q}}{|\cdot-x|^{n}}\right\Vert _{L^{1,\infty}\left(\mathbb{R}^{n}\right)}\leq cw(x)\qquad\text{a.e. }x\in\mathbb{R}^{n}
\]
but we don't know whether is sufficient or not.

To end this section we show that  can go further and prove a negative result for possible $L\log L$ type estimates.

\begin{observation}
Let $1<p<\infty$, and let $\Phi(t)=t\,\log(e+t)^{\alpha} $, $\alpha>0$.  Then  we can find $w\in A_{p}\setminus A_{1}$
and $f$ such that there's no $c>0$ for which 
\begin{equation} \label{LlogL-false}
\left|\left\{ x\in\mathbb{R}^{n}\,:\,wM\left(\frac{f}{w}\right)>t\right\} \right|\leq c\int_{\mathbb{R}^{n}}\Phi\left(\frac{|f(x)|}{t}\right)dx.
\end{equation}

\end{observation}
\begin{proof} As above we do it for the case $n=1$. 
We assume the contrary, namely there is a finite constant $c>0$ such that (\ref{LlogL-false}) holds for any nonnegative $f$. 
Let $f=\chi_{(0,1)}$. For this choice of $f$ the right hand side of (\ref{LlogL-false}) equals  $\Phi\left(\frac{1}{t}\right)$ and we have that
\[
\sup_{t>0} \frac{1}{\Phi\left(\frac{1}{t}\right)}\left|\left\{ x\in\mathbb{R}\,:\,wM\left(\frac{\chi_{(0,1)}}{w}\right)>t\right\} \right| <\infty. 
\]
Choose again the $A_p$ weight $w(x)=|x|^{-\delta(1-p)}$ with $\delta\in\left(0,\min\left\{ 1,\frac{1}{p-1}\right\} \right)$. 
Proceeding and using the same notation as in the proof of Observation \ref{Thm19} we have that
%
\[
\begin{split}\sup_{t>0}\frac{1}{\Phi\left(\frac{1}{t}\right)}\left|\left\{ x\in\mathbb{R}\,:\,wM\left(\frac{\chi_{(0,1)}}{w}\right)>t\right\} \right| & \geq c\,\sup_{0<t<1}\frac{1}{\Phi\left(\frac{1}{t}\right)}\left[\left(\frac{1}{t}\right)^{\frac{1}{\beta}}-1\right]\\
 & =c\,\sup_{0<t<1}\frac{t}{\log\left(e+\frac{1}{t}\right)^{\alpha}}\left[\left(\frac{1}{t}\right)^{\frac{1}{\beta}}-1\right] =\infty.
\end{split}
\]
since $\beta \in (0,1)$.

\end{proof}

\bibliographystyle{plain}
\bibliography{refs}

\end{document}